\documentclass[11pt]{article}
\usepackage{amsmath,amsfonts,latexsym, xspace,amsthm,graphicx, amssymb}
\usepackage{enumerate}

\usepackage{url}

\usepackage{mathtools}
\mathtoolsset{showonlyrefs}

\newcommand{\al}[1]{
  \begin{align}
  #1
  \end{align}
}

\newcommand{\E}[1]{\mathbb{E} \left[#1\right]}

\newcommand{\Bin}[2]{\text{Bin}\left(#1, #2\right)}

\newcommand{\Var}[1]{\text{Var}\left(#1\right)}

\def\AA{\mathcal{A}}
\def\BB{\mathcal{B}}
\def\CC{\mathcal{C}}
\def\DD{\mathcal{D}}
\def\EE{\mathcal{E}}

\def\GG{\mathcal{G}}
\def\HH{\mathcal{H}}

\def\MM{\mathcal{M}}
\def\NN{\mathcal{N}}
\def\PP{\mathcal{P}}

\def\RR{\mathcal{R}}
\def\SS{\mathcal{S}}
\def\TT{\mathcal{T}}

\def\ol{\overline}

\def\N{\mathbb{N}}

\def\a{\alpha}
\def\b{\beta}
\def\d{\delta}

\def\e{\varepsilon}
\def\f{\phi}

\def\G{\Gamma}

\def\s{\sigma}

\def\t{\tau}

\def\om{\omega}
\def\OM{\Omega}
\newcommand\Prob[1]{{\mbox{Pr}\left\{#1\right\}}}

\newtheorem{theorem}{Theorem}

\newtheorem{lemma}[theorem]{Lemma}
\newtheorem{corollary}[theorem]{Corollary}

\newcommand{\bfrac}[2]{\left({\frac{#1}{#2}}\right)}

\title{On Hamilton cycles in Erd\H{o}s--R\'enyi subgraphs of large graphs}

\author{Tony Johansson\thanks{Research supported in part by the Knut and Alice Wallenberg foundation.} \\[.2cm]
{\small  Uppsala University} \\
{\small Uppsala, Sweden}}

\begin{document}
\maketitle

\begin{abstract}
Given a graph $\G = (V, E)$ on $n$ vertices and $m$ edges, we define the Erd\H{o}s-R\'enyi graph process with host $\G$ as follows. A permutation $e_1,\dots,e_m$ of $E$ is chosen uniformly at random, and for $t\leq m$ we let $\G_t = (V, \{e_1,\dots,e_t\})$. Suppose the minimum degree of $\G$ is $\d(\G) \geq (1/2 + \e)n$ for some constant $\e > 0$. Then with high probability\footnote{An event $\EE_n$ holds {\em with high probability (whp)} if $\Prob{\EE_n} \to 1$ as $n\to\infty$.}, $\G_t$ becomes Hamiltonian at the same moment that its minimum degree becomes at least two. 

Given $0\leq p\leq 1$ we let $\G_p$ be the Erd\H{o}s-R\'enyi subgraph of $\G$, obtained by retaining each edge independently with probability $p$. When $\d(\G)\geq (1/2 + \e)n$, we provide a threshold function $p_0$ for Hamiltonicity, such that if $(p-p_0)n\to -\infty$ then $\G_p$ is not Hamiltonian whp, and if $(p-p_0)n\to\infty$ then $\G_p$ is Hamiltonian whp.
\end{abstract}

\section{Introduction}\label{sec:intro}

Given a Hamiltonian graph $\G$ on $n$ vertices and $m$ edges, pick a random ordering $e_1,\dots,e_{m}$ of its edges, and let $\G_t$ (or $\G_{n, t}$) be the subgraph consisting of $e_1,\dots,e_t$. Let $\d(G)$ denote the minimum degree of a graph, and define
\al{
  \t_2 & = \min\{t : \d(\G_t) \geq 2\}, \\
  \t_H & = \min\{t : \G_t\text{ contains a Hamilton cycle}\}.
}
It is trivial to see that $\t_2 \leq \t_H$. A celebrated result, independently shown by Ajtai, Koml\'os and Szemer\'edi \cite{AKS} and by Bollob\'as \cite{Bollobas}, is the fact that $\t_2 = \t_H$ with high probability in the case $\G = K_n$. We generalize this result to a large class of graphs.
\begin{theorem}\label{thm:hit}
Let $\b > 1/2$ be constant and suppose $\d(\G) \geq \b n$. Then $\t_H = \t_2$ whp.
\end{theorem}
Dirac's theorem \cite{Dirac} states if $\G$ has $n$ vertices and $\d(\G)\geq n/2$, then $\G$ contains a Hamilton cycle. In many random graph models, it is enough that the random graph has constant minimum degree, see e.g. \cite{BohFri, FPPR, RW}. The most striking example of this phenomenon remains the Erd\H{o}s-R\'enyi graph, and the connection between Hamiltonicity and minimum degree $2$ has been well studied when $\G = K_n$. Alon and Krivelevich \cite{AK} recently proved that the probability that $G_{n, p}$ contains no Hamilton cycle is $(1 + o(1))\Prob{\d(G) < 2}$ for all values of $p$. It is known that at the moment the graph process reaches minimum degree $2k$, $G_{n, t}$ contains $k$ edge disjoint Hamilton cycles \cite{BolFri,KriSam,KKO}. Briggs et al \cite{BFKLS} showed that the edges upon insertion can be coloured in one of $k$ colours, with each colour class containing a Hamilton cycle at the moment the minimum degree reaches $2k$.

Given a graph $\G = (V, E)$ and $0\leq p\leq 1$, we also define the Erd\H{o}s-R\'enyi subgraph $\G_p$ (or $\G_{n, p}$) as the random subgraph of $\G$ obtained by independently retaining each edge with probability $p$. For $\G = K_n$ (so $\G_p = G_{n, p}$), we have \cite{Posa, Korsunov, KS} that if $p = (\log n + \log\log n + c_n) / n$, then
$$
\lim_{n\to\infty}\Prob{G_{n, p} \text{ is Hamiltonian}} = \left\{\begin{array}{ll}
0, & c_n \to-\infty, \\
e^{-e^{-c}}, & c_n\to c, \\
1, & c_n \to\infty.
\end{array}\right.
$$
We show analogous results for graphs with $\d(\G) \geq (1/2+\e)n$ below.

Traditionally, research on random graphs has mostly been concerned with random subgraphs of specific graphs such as $K_n$ or the complete bipartite graph $K_{n, n}$ (see e.g. \cite{Frieze}). Research on Erd\H{o}s-R\'enyi subgraphs of graphs $\G$ with large minimum degree was initiated by Krivelevich, Lee and Sudakov \cite{KLS}, who among other things showed that if $\d(\G) \geq n/2$ and $p = C\log n / n$ for some large constant $C$, then $\G_p$ is Hamiltonian whp. This article determines the exact value of $C$ when $\d(G) \geq (1/2+\e) n$ for some constant $\e > 0$. The same authors showed \cite{KLS2} that if $\d(\G) \geq k$ for any $k$ tending to infinity with $n$, and $pk$ tends to infinity with $n$, then $\G_p$ contains a cycle of length $(1 - o_k(1))k$ whp. Riordan \cite{Riordan} subsequently gave a short proof of this result, and Glebov, Naves and Sudakov \cite{GNS} showed that if $p \geq (\log k + \log\log k + \om)/k$ for some $\om$ tending to infinity then $\G_p$ contains a cycle of length $k+1$ whp.

A related topic is {\em resiliency}. A graph $G$ is said to be {\em $\a$-resilient} with respect to a property held by $G$ if the graph $G\setminus H$ also has the property for any spanning subgraph $H\subseteq G$ where $d_H(v) \leq \a d_G(v)$ for all $v$. Nenadov, Steger and Truji\'c \cite{NST} and independently Montgomery \cite{Montgomery} showed among other things that $G_{n, t}$ is $(1/2 - o(1))$-resilient with respect to Hamiltonicity at the moment the minimum degree reaches $2$. Very recently, Condon et al \cite{CDKKO} showed a resiliency version of P\'osa's theorem in $G_{n, p}$.

Many of these results have analogues for perfect matchings. Most pertinent to the result presented here is that of Glebov, Luria and Simkin \cite{GLS}, who showed that if $\G$ is a $d$-regular bipartite graph with $d = \OM(n)$, then whp $\G_t$ obtains a perfect matching at the same moment it loses its last isolated vertex.

The author and Frieze \cite{FJ} considered graphs $\G$ with $\d(\G) \geq (1/2+\e)n$ and studied the Hamiltonicity of random $k$-out subgraphs of $\G$. In the context of $k$-out graphs with $k = O(1)$, the positive constant $\e$ is needed for the random subgraph to be connected whp. We note that this is not the case for Erd\H{o}s-R\'enyi graphs, and that it is still unknown whether the hitting time result presented here can be extended to all $\G$ with $\d(\G)\geq n/2$.

Theorem \ref{thm:hit} implies that finding a threshold for Hamiltonicity is simply a matter of finding the threshold for minimum degree $2$. 
\begin{theorem}\label{thm:2threshold}
Let $\e > 0$ be constant and suppose $\om = o(\log\log n)$ tends to infinity arbitrarily slowly with $n$. Suppose $(\G_n)_{n\in \N}$ is some graph sequence where $\G_n = (V_n, E_n)$ has $n$ vertices and minimum degree $\d(\G_n) \geq \e n$. For each $n$ define $p_0(n)$ as the unique solution in the interval $0 \leq p \leq 1$ to the equation
\begin{equation}\label{eq:p0def}
\sum_{v\in V_n} (1-p)^{d_\G(v)}\log n = 1.
\end{equation}
Then
\al{
  \Prob{\d(\G_{p_0 - \om / n}) \geq 2} & = o(1), \\
  \Prob{\d(\G_{p_0 + \om / n}) \geq 2} & = 1-o(1).
}
\end{theorem}
If a large number of vertices of $\G_n$ have degree equal to the minimum degree $\b n$ for all large $n$, we obtain
$$
p_0(n) = \frac{\log n + \log\log n + c(\G_n)}{\b n}
$$
for some bounded $c(\G_n)$, which shows that the well-known threshold function for $\G = K_n$ (where $\b = 1$ and $c(K_n) \to 0$) is scaled by a factor of $\b^{-1}$. The precise statement is the following.
\begin{corollary}\label{cor:Hthreshold}
Let $\e > 0$. Suppose $\b > 1/2$ is constant and that $(\G_n)$ is a graph sequence with $\d(\G_n) = \b n$, where at least $\e n$ vertices of $\G_n$ have degree $\b n$, and let $p = (\log n + \log\log n + c_n)/\b n$. Then
$$
\lim_{n\to\infty} \Prob{\G_{n, p} \text{ is Hamiltonian}} = \left\{\begin{array}{ll}
0, &  c_n \to -\infty, \\
1, &  c_n \to \infty.
\end{array}\right.
$$
Furthermore, if $\G_n$ is $\b n$--regular and $c_n\to c$ for some constant $c$, then
$$
\lim_{n\to\infty}\Prob{\G_{n, p} \text{ is Hamiltonian}} = e^{-e^{-\b c}}.
$$
\end{corollary}
The corollary follows from calculating the probability that no vertex has degree less than $2$. We omit some of the calculations here; see e.g. \cite[Theorem 3.1]{FK} for a proof in the $\G = K_n$ case.

The paper is laid out as follows. In Section \ref{sec:randomgraph} we define the random graph $\G_{\t_2}$ stopped at the moment the minimum degree reaches two, as well as an auxiliary subgraph $G_{\t_2}$. Section \ref{sec:posa} discusses P\'osa's rotation-extension technique, and Section \ref{sec:threshold} is devoted to proving Theorem \ref{thm:2threshold}. In Section \ref{sec:hit} (and its many subsections) we prove Theorem \ref{thm:hit}.

We use the asymptotic notation $O, \OM, o$, with the convention that $f(n) = \OM(g(n))$ and $f(n) = O(g(n))$ both require $f(n)$ to be nonnegative. All logarithms are taken in the natural basis. 

\section{The random graph model}\label{sec:randomgraph}

Let $\e > 0$ be constant and suppose $\G = (V, E)$ is a graph with minimum degree $\d(\G) \geq (1/2 + \e) n$ and $m$ edges. Suppose $(e_1,\dots,e_m)$ is some permutation of the edges, chosen uniformly at random. We define $\G_t = (V, \{e_1,\dots,e_t\})$ for all $1\leq t\leq m$. Define $\t_2$ as the smallest $t$ for which $\G_t$ has minimum degree at least $2$.

We let $q = \om / \log n$ for some $\om = o(\log\log n)$ tending to infinity arbitrarily slowly with $n$. Upon insertion, edges are independently coloured red with probability $q$ and blue with probability $1 - q$. Let $\G_t^b$ denote the blue subgraph, i.e. the subgraph consisting of all blue edges.

Set $\s = 1/100$. We define $\mathrm{SMALL}$ as the set of vertices with degree less than $\s\log n$  in $\G_{\t_2}$, and $\mathrm{LARGE} = V \setminus \mathrm{SMALL}$. We also define $\mathrm{MEDIUM}$ as the set of vertices with degree less than $\s\log n$ in $\G_{\t_2}^b$. Let $G_{\t_2}\subseteq \G_{\t_2}$ be the graph consisting of the blue edges, along with all red edges with at least one endpoint in $\mathrm{MEDIUM}$. Note that $\mathrm{SMALL}\subseteq \mathrm{MEDIUM}$. By design, $G_{\t_2}$ and $\G_{\t_2}$ agree on any property concerning only vertices of degree at most $\s\log n$ and their incident edges.

The following lemma will allow us to switch between $\G_t$ and $\G_p$. A property $\PP$ is {\em increasing} if for any $G\in \PP$ and any graph $H$, we have $G\cup H\in \PP$, and {\em decreasing} if $G\in \PP$ implies $G\setminus H\in \PP$ for any $H$. A property is {\em monotone} if it is either increasing or decreasing.
\begin{lemma}\label{lem:ptot}
Let $\PP$ be a graph property, and $\G$ a graph with $m$ edges. If $p = t / m$ then
$$
\Prob{\G_t \in \PP} \leq 10 t^{1/2} \Prob{\G_p \in \PP}.
$$
If $\PP$ is a monotone property and $t = o(m)$ tends to infinity with $n$, then
$$
\Prob{\G_t \in \PP} \leq 10\Prob{\G_p \in \PP}.
$$
\end{lemma}
This result is well known when $\G = K_n$ (see e.g. \cite[Lemmas 1.2, 1.3]{FK}), and the straightforward generalization to general dense $\G$ is omitted here.

\section{Rotation and extension}\label{sec:posa}

We define rotations of longest paths, as introduced by P\'osa \cite{Posa}. Suppose $G = (V, E)$ is a graph containing no Hamilton cycle, and let $P = (v_0, v_1,\dots,v_\ell)$ be a path of maximal length in $G$. If $i \leq \ell - 2$ and $\{v_\ell, v_i\} \in E$, then the path $P' = (v_0,v_1,\dots,v_{i-1}, v_i,v_\ell, v_{\ell-1},\dots,v_{i+2},v_{i+1})$ is also a path of maximal length. We say that $P'$ is obtained from $P$ by rotation with $v_0$ fixed. Let $EP(v_0)$ be the set of endpoints, other than $v_0$, that appear as a result of rotations with $v_0$ fixed.

We note that if $G$ is connected, then there is no edge between $v_0$ and $EP(v_0)$ for any $v_0$, as this forms a cycle which may be extended to form a longer path, contradicting the maximality of $P$. Given a graph $G = (V, E)$, we write
$$
N_G(S) = \{v\in V\setminus S \mid \exists u\in S : \{u, v\}\in E\}.
$$
We will use the following lemma of P\'osa \cite[Lemma 1]{Posa}.

\begin{lemma}\label{lem:posa}
Suppose $G$ contains no Hamilton cycle and let $P = (v_0,\dots,v_\ell)$ be a path of maximum length in $G$. Then
$$
|N_G(EP(v_0))| < 2|EP(v_0)|.
$$
\end{lemma}

\section{The threshold}\label{sec:threshold}

Suppose the underlying graph sequence is $(\G_n)_{n\in \N}$, where $\G_n$ has $n$ vertices $V_n$. We define $p_0 = p_0(n)$ as the unique solution to
$$
\sum_{v\in V_n} (1-p)^{d_\G(v)}\log n = 1.
$$
This exists as the left-hand side equals $n\log n > 1$ for $p = 0$, zero for $p = 1$, and is strictly decreasing in $p$. We note that
$$
1 = \sum_{v\in V_n} (1-p_0)^{d_\G(v)} \leq n\exp\left\{-p_0\d(\G_n)\right\},
$$
so $p_0 \leq \d(\G_n)^{-1} \log n \leq 2\frac{\log n}{n}$. We also have
$$
\sum_{v\in V_n} \left(1 - \frac{\log n}{n}\right)^{d_\G(v)}\log n  \geq n\left(1-\frac{\log n}{n}\right)^n\log n \sim \log n,
$$
so $p_0 \geq \frac{\log n}{n}$. The following bounds, both of which use $\d(\G)\geq (1/2 + \e)n$, will be used frequently:
\begin{align}
\sum_{v\in V_n} (1-p)^{d_\G(v)} \geq \frac{e^{\om/2}}{\log n}, & & p = p_0 - \frac{\om}{n}, \label{eq:ombelow} \\
\sum_{v\in V_n} (1-p)^{d_\G(v)} \leq \frac{e^{-\om/2}}{\log n}, & & p = p_0 + \frac{\om}{n}. \label{eq:omabove}
\end{align}
Here \eqref{eq:ombelow} follows from summing the following inequality over $v$:
\begin{multline}
(1-p)^{d_\G(v)} = (1-p_0)^{d_\G(v)} \bfrac{1-p}{1-p_0}^{d_\G(v)} \\
 \geq (1-p_0)^{d_\G(v)}\left(1 + \frac{\om}{n(1-p_0)}\right)^{\d(\G)} \geq (1-p_0)^{d_\G(v)} e^{\om / 2},
\end{multline}
and \eqref{eq:omabove} follows similarly.

The following lemma is easily generalized to smaller $\b$, and we insist that $\b > 1/2$ for notational convenience only.
\begin{lemma}\label{lem:threshold}
Let $\b > 1/2$ be constant. Suppose $\om = o(\log\log  n)$ tends to infinity arbitrarily slowly with $n$, and suppose $\d(\G_n) \geq \b n$ for all $n$, and that $\G_n$ has $m$ edges. Let
$$
T = \left(p_0 - \frac{\om}{n}\right)m, \quad T' = \left(p_0 + \frac{\om}{n}\right)m.
$$
Then
\al{
  \Prob{\d(G_T) \geq 2} & = o(1), \\
  \Prob{\d(G_{T'}) \geq 2} & = 1 - o(1).
}
\end{lemma}

\begin{proof}
Let $p = p_0 - \om/n$. Then
\begin{align}
  \Prob{d_G(v) < 2} & = (1-p)^{d_\G(v)} + d_\G(v) p(1-p)^{d_\G(v) - 1} \nonumber \\
  & = pd_\G(v)(1-p)^{d_\G(v)}(1 + o(1)). \label{eq:prob2}
\end{align}
Here we used the facts that $d_\G(v) = \OM(n)$ and $p = \Theta\bfrac{\log n}{n}$. Let $I_v$ be the indicator variable for $\{d_G(v) < 2\}$, and write $X_n = \sum_v I_v$. We then have
$$
\E{X_n} = (1 + o(1))\sum_{v} pd_\G(v)(1-p)^{d_\G(v)}.
$$

{\bf Lower bound: $p = p_0 + \om / n$.} In this case, as $pd_\G(v) \leq 2\log n$,
\begin{equation}\label{eq:Markov}
\E{X_n} \leq 3\log n \sum_v (1-p)^{d_\G(v)} \leq 3e^{-\om/2},
\end{equation}
by \eqref{eq:omabove}. Markov's inequality implies that $X_n = 0$ whp, and Lemma \ref{lem:ptot} shows that $\d(G_{T'}) \geq 2$ whp, as this is a monotone property.

{\bf Upper bound: $p = p_0 - \om / n$.} We apply the second moment method. Using \eqref{eq:ombelow}, similarly to \eqref{eq:Markov} we have $\E{X_n} \geq (1-o(1))e^{\om/2}$, and in particular $\E{X_n}$ tends to infinity with $n$. We also have
$$
\E{X_n(X_n - 1)} = \sum_{u\neq v} \E{I_uI_v}.
$$
Firstly, if $\{u, v\}\notin E(\G)$ then $I_u$ and $I_v$ are independent and $\E{I_uI_v} = \E{I_u}\E{I_v}$. If $\{u, v\}\in E(\G)$ then
\al{
  \E{I_uI_v} & = p(1-p)^{d_\G(u) + d_\G(v)} + (d_\G(u)-1)(d_\G(v)-1)p^2(1-p)^{d_\G(u) + d_\G(v) - 3} \\
  & = d_\G(u)d_\G(v)p^2(1-p)^{d_\G(u) + d_\G(v)}\left(1 + O(n^{-1})\right) \\
  & = \E{I_u}\E{I_v}(1+O(n^{-1})).
}
We then have 
$$
\E{X_n^2} = (1 + o(1)) \sum_{u\neq v} \E{I_u}\E{I_v} = (1 + o(1))\left(\sum_v pd_\G(v)(1-p)^{d_\G(v)}\right)^2.
$$
It follows that $\Var{X_n} = o(\E{X_n^2})$, and Chebyshev's inequality implies that $X_n > 0$ whp, and so $\d(\G_p) < 2$ whp. This is a monotone event, so $\d(\G_T) < 2$ also holds whp by Lemma \ref{lem:ptot}.
\end{proof}

\section{Proof of Theorem \ref{thm:hit}}\label{sec:hit}

We set up the main calculation. We fix the constants $K = 10, \s = 1/100$ and $\a = e^{-2000}$, and also fix some $\om = o(\log\log n)$ tending to infinity arbitrarily slowly with $n$. Define the following events concerning $\G_{\t_2}$ (see Section \ref{sec:randomgraph} for definitions concerning our random graphs).
\al{
  \HH & = \{\text{$\G_{\t_2}$ is Hamiltonian}\}, \\
  \PP_\ell & = \{\text{$\G_{\t_2}$ is not Hamiltonian, and its longest path has $\ell$ vertices}\}, \\
  \SS & = \{|\mathrm{SMALL}| \leq n^{0.1}\}, \\
  \EE_1 & = \left\{\text{in $\G_{\t_2}$, every $|S|\leq 6\a n$ has $e(S) \leq \frac{\s\log n}{K}|S|$}\right\}, \\
  \EE_2 & = \{\text{$\G_{\t_2}$ contains no path of length $\leq 4$ between vertices of $\mathrm{SMALL}$}, \\
  & \ \ \ \ \ \ \text{and no cycle of length $\leq 4$ intersects $\mathrm{SMALL}$}\}, \\
  \EE & = \EE_1\cap \EE_2, \\
  \NN & = \EE \cap \SS.
}
We define $\HH', \PP_\ell'$, etc., as the corresponding events with $G_{\t_2}$ replacing $\G_{\t_2}$, noting that the set of vertices of degree less than $\s\log n$ is unchanged. We further define the following events concerning $G_{\t_2}$.
\al{
  \CC' & = \{\text{$G_{\t_2}$ is connected}\}, \\
  \MM' & = \left\{|\mathrm{MEDIUM}| \leq \frac{\om n}{\log\log n}\right\}, \\
  \RR' & = \left\{e(\G_{\t_2}) - e(G_{\t_2}) \geq \frac{\om n}{2}\right\}, \\
  \GG_\ell' & = \PP_\ell' \cap \NN' \cap \CC' \cap \MM' \cap \RR'.
}
We note that $\NN\subseteq \NN'$, as the properties involved are either decreasing or only concern vertices of degree less than $\s\log n$. We finally define, with $T = (p_0 - \om / n)m$ as in Section \ref{sec:threshold},
\al{
  \TT & = \{\t_2 \geq T\}, \\
  \AA & = \bigcup_{\ell = 1}^{n} (\PP_\ell\cap \PP_\ell').
}
In words, $\AA$ is the event that the longest path in $G_{\t_2}$ has the same length (in terms of the number of vertices) as the longest path in $\G_{\t_2}$. The following lemma is proved in Section \ref{sec:expansion}.
\begin{lemma}\label{lem:expansion}
If $\EE$ occurs, then any set $S$ of at most $\a n$ vertices satisfies $|N(S)| \geq 2|S|$ in $\G_{\t_2}$ and in $G_{\t_2}$.
\end{lemma}
We will show in the upcoming two sections that
\begin{equation}\label{eq:lb}
\Prob{\AA \mid \PP_\ell\cap \NN\cap\TT} \geq \exp\left\{-O\bfrac{\om n}{\log n}\right\},
\end{equation}
and that
\begin{equation}\label{eq:AcapB}
\Prob{\AA \cap \PP_\ell\cap \NN\cap\TT} \leq \exp\left\{-\OM\bfrac{\om n}{\log\log n}\right\},
\end{equation}
from which we conclude that
$$
\Prob{\PP_\ell\cap \NN\cap\TT} = \frac{\Prob{\AA\cap \PP_\ell\cap \NN\cap\TT}}{\Prob{\AA\mid\PP_\ell\cap \NN\cap\TT}} \leq \frac{e^{-\OM(\om n/\log\log n)}}{e^{-O(\om n / \log n)}} = o(n^{-1}).
$$
We then have (as $\Prob{\TT} = 1 - o(1)$ by Lemma \ref{lem:threshold}),
$$
  \Prob{\ol \HH} \leq \Prob{\ol{\NN}\cup \ol{\TT}} + \sum_{\ell} \Prob{\PP_\ell\cap \NN\cap\TT} = \Prob{\ol{\NN}} +  o(1).
$$
To finish the argument, we show in Section \ref{sec:props} that
\begin{equation}\label{eq:N}
\Prob{\NN} = 1 - o(1).
\end{equation}

\subsection{Proof of \eqref{eq:lb}}\label{sec:lb}

Suppose $\G_{\t_2} = G$, where $G$ is a graph with some longest path (or Hamilton cycle) $P$, on edges $f_1,\dots,f_\ell$ with $\ell \leq n$. If each edge of $P$ is coloured blue then $P$ must also appear in $G_{\t_2}$. So, as $q = \om / \log n$,
\al{
  \Prob{\AA\mid \G_{\t_2} = G} & \geq \left(1 - \frac{\om}{\log n}\right)^\ell = \exp\left\{-O\bfrac{\om n}{\log  n}\right\}.
}
As the event $\PP_\ell\cap \NN$ is of the form $\{\G_{\t_2}\in \GG\}$ for some class of graphs $\GG$, \eqref{eq:lb} follows.

\subsection{Proof of \eqref{eq:AcapB}}\label{sec:boosterbuilding}

By the discussion in Section \ref{sec:hit}, we have $\NN\subseteq \NN'$. Recall $\GG_\ell' = \PP_\ell'\cap \NN'\cap \CC'\cap \MM'\cap \RR'$. Then as $\NN'\subseteq \EE'$,
$$
\AA\cap \PP_\ell\cap\NN\cap\TT \subseteq (\AA\cap \GG_\ell') \cup (\ol{\CC'}\cap \EE') \cup (\ol{\MM'\cap \RR'}\cap \TT).
$$
We also have $\AA\cap \PP_\ell = \AA\cap \PP_\ell'$, so
\al{
  & \Prob{\AA\cap \PP_\ell\cap \NN\cap\TT}\\
  \leq\ & \Prob{\AA\cap \GG_\ell'} + \Prob{\ol{\CC'} \cap \EE'} + \Prob{\ol{\MM'\cap\RR'} \cap \TT}\\
  \leq\ & \Prob{\PP_\ell\mid \GG'} + \Prob{\ol{\CC'} \cap \EE'} + \Prob{\ol{\MM'}\cap \TT} + \Prob{\ol{\RR'}\cap \MM' \cap \TT}.  \label{eq:manyevents}
}
In this section we show that the first term is at most $e^{-\OM(\om n)}$, while the other terms are postponed for Section \ref{sec:otherbounds}. 

So we condition on $\GG_\ell' = \PP_\ell'\cap \NN'\cap \CC'\cap \MM'\cap \RR'$. Then $G = G_{\t_2}$ is a connected graph such that $|N_G(S)| \geq 2|S|$ for any $|S|\leq \a n$ (see Lemma \ref{lem:expansion}), and there exists a set $|L| = n - o(n)$, such that $\G_{\t_2}$  is obtained from $G$ by randomly adding $r \geq \om n / 2$ edges from $\G$ with both endpoints in $L$. Let $R$ denote the set of red edges fully contained in $L$. The longest path of $G$ has length $\ell$, and we will show that it is very unlikely that adding the edges $R$ does not increase the length of the longest path.

Let $P$ be a longest path in $G$, and let $x, y$ be its two endpoints. Let $EP(x)$ be the set of opposite endpoints obtainable from $P$ by rotations with $x$ fixed. By Lemma \ref{lem:expansion} we have $|N_G(S)| \geq 2|S|$ whenever $|S| \leq \a n$, so Lemma \ref{lem:posa} implies $|EP(x)| \geq \a n$. Let $EP$ be the set of all endpoints of longest paths in $G$. The total number of boosters in $G$ is
$$
b = \frac12 \sum_{x\in EP}\sum_{y\in EP(x)} [\{x, y\}\in E(\G)],
$$
where we write $[\BB] = 1$ if the statement $\BB$ is true, and $0$ otherwise. We divide the set $R$ of random edges into two parts $R_1\cup R_2$ of as equal size as possible. We use $R_1$ to build $\OM(n^2)$ boosters, in case $G$ does not already have $\OM(n^2)$ boosters.
\begin{lemma}\label{lem:boosterbuilding}
With probability at least $1 - e^{-\OM(\om n)}$, $G\cup R_1$ either has a path of longer length than $\ell$ (or is Hamiltonian), or it has $\OM(n^2)$ boosters.
\end{lemma}
\begin{proof}
Assign some arbitrary order $R_1 = \{f_1,\dots,f_{r/2}\}$ to the edges of $R_1$. Let $L = \mathrm{LARGE}\setminus \mathrm{MEDIUM}$. We can treat the $f_i$ as independent uniform edges in $E(\G)\cap L^2$, as doing so only introduces repetitions which decreases the probability of producing many boosters.

Let $P$ be a longest path on vertex set $U$, and let $EP$ be the set of endpoints of paths spanning $U$. Let $EP_L = EP\cap L$. Suppose first that there are at least $(\e n)^2$ edges in $\G$ between $EP_L$ and $\ol U \cap L$. Adding any of these edges extends the path. Each $f_i$ has probability at least $\e$ of landing in this set, so the probability that the path is not extended by such an edge is at most $(1 - \e)^{r/2} \leq e^{-\OM(\om n)}$.

Suppose that there are less than $(\e n)^2$ edges between $EP_L$ and $\ol U\cap L$. As $|\mathrm{MEDIUM}| = o(n)$, there are $o(n^2)$ edges incident to $\mathrm{MEDIUM}$. So at most $2\e n$ vertices of $EP_L$ can have more than $\e n / 2$ edges to $\ol U$. Say that $y\in EP$ is {\em good} if it is in $L$ and has at least $(1+\e)n/2$ edges to $U$ in $\G$. For any $x$, the set $EP(x)$ of opposite endpoints must contain at least $(\a - 2\e)n - o(n) \geq \a n / 2$ good vertices.

We aim to show that in $G\cup R_1$, with probability $1 - e^{-\OM(\om n)}$, either there is a path of longer length than $\ell$ (or is Hamiltonian), or all good endpoints are incident to at least $\e n/8$ boosters. Say that $x\in EP_L$ is {\em settled} if there are at least $\e n / 8$ vertices $y\in EP_L(x)$ such that $\{x, y\}\in E(\G)$, and {\em unsettled} otherwise.

For now, we consider $G_0 = G$, i.e. with no edges of $R_1$ added. Suppose $x$ is good and unsettled. Let $y$ be a good vertex of $EP(x)$, and let $Q = (x = v_0,v_1,\dots,v_\ell = y)$ be a longest path on $U$ between $x$ and $y$. As both $x$ and $y$ are good, there are at least $\e n / 2$ indices $i$ such that $\{x, v_{i + 1}\}, \{y, v_i\} \in E(\G)$. As $\mathrm{MEDIUM}$ has size $o(n)$, there must be at least $\e n / 4$ such indices where $v_i$ and $v_{i+1}$ are both in $L$. As $x$ is unsettled, at most $\e n / 8$ such indices have $v_{i + 1}\in EP(x)$, and for each such index it follows that $\{y, v_i\}\notin E(G)$, as otherwise a rotation around $\{y, v_i\}$ would contradict $v_{i + 1}\notin EP(x)$. We conclude that there are at least $\e n / 8$ indices $i$ such that $\{x, v_{i+1}\} \in E(\G), \{y, v_i\} \in E(\G)\setminus E(G)$, and all four vertices are in $L$. Let $A_0(x, y)$ be the set of such semiboosters $\{y, v_i\}$, and $B_0(x, y)$ the set of such $v_{i+1}$. Repeat this for all good $y\in EP(x)$ (if there is more than one longest path between $x$ and $y$, pick one arbitrarily). Let $A_0(x)$ be the union of $A_0(x, y)$ over all good $y\in EP(x)$. As there are at least $\a n / 2$ good $y\in EP(x)$, and each edge in $A_0(x, y)$ is incident to $y$, we have
$$
|A_0(x)| = \sum_{\substack{y\in EP(x) \\ \text{good}}} |A_0(x, y)| \geq \frac{\a n}{2} \times \frac{\e n}{8} = \frac{\a\e}{16}n^2.
$$
The sets $B_0(x, y)\subseteq U\cap N_\G(x)$ are not necessarily disjoint for different $y$. However, each $z\in U\cap N_\G(x)$ appears in $B_0(x, y)$ for at most $|EP(x)| < n$ different $y$. 

Each $e\in A_0(x)$ is associated with some vertex $z = z(e)\in U\cap N_\G(x)$ such that adding $e$ implies that $\{x, z\}$ is a new booster (or extends the path if $\{x, z\}$ is in $G$). Let $G_1 = G_0\cup \{f_1\}$. If $f_1\in A_0(x)$, then the booster $\{x, z(f_1)\}$ has been added. Any $e\in A_0(x) \setminus \{f_1\}$ with $z(e) = z(f_1)$ may no longer be a semibooster, and we remove such $e$ from our set of semiboosters.

In general, given $A_j(x)$, we reveal $f_{j + 1}$. If $f_{j + 1}\in A_j(x)$ (a {\em success}) we set $A_{j + 1}(x) = A_j(x) \setminus z^{-1}(f_{j + 1})$. If $f_{j + 1}\notin A_j(x)$ (a {\em failure}), we let $A_{j + 1}(x) = A_j(x)$. If revealing $f_1,\dots,f_j$ results in $s$ successes, then
$$
|A_j(x)| \geq |A_0(x)| - sn \geq \frac{\a\e}{16}n^2 - sn.
$$
As long as there are less than $\a \e n / 32$ successes, there are at least $\a\e n^2 / 32$ semiboosters left in $A_j(x)$. In this case, $f_{j + 1}$ has probability at least $|A_j(x)| / \binom{n}{2} \geq \a \e / 16$ of succeeding. So the probability that adding all of $R_1$ results in less than $\a\e n / 32$ successes is bounded by
$$
\Prob{\Bin{\frac{r}{2}}{\frac{\a \e}{16}} < \frac{\a\e n}{32}} = e^{-\OM(\om n)}.
$$
So with probability $1 - e^{-\OM(\om n)}$, $x$ is incident to at least $\a \e n / 16$ boosters for $G\cup R_1$. Repeating this for all good and unsettled $x$, and recalling that a settled $x$ already is incident to many boosters, the probability that some good $x$ is incident to less than $\a \e n /16$ boosters in $G\cup R_1$ is at most $ne^{-\OM(\om n)} = e^{-\OM(\om n)}$.  This shows that $G\cup R_1$, if still in $\PP_\ell$, has at least
$$
\frac12 \times \frac{\a n}{2} \times \frac{\a\e}{16} n = \frac{\a^2\e}{64} n^2
$$
boosters with probability $1 - e^{-\OM(\om n)}$.
\end{proof}

Suppose $G\cup  R_1$ has $\OM(n^2)$ boosters. We expose $R_2$. The probability that none of the $\OM(n^2)$ boosters are in $R_2$ is $e^{-\OM(\om n)}$. As we condition on $\GG_\ell' \subseteq \CC'$, adding a booster extends the length of the longest path or forms a Hamilton cycle. We conclude that
$$
\Prob{\PP_\ell \mid \GG_\ell'} \leq e^{-\OM(\om n)}.
$$

\subsection{Properties of $G_{\t_2}$}\label{sec:otherbounds}

Recall from \eqref{eq:manyevents} that it remains to show that $\Prob{\ol{\CC'}\cap \EE'}$, $\Prob{\ol{\MM'}\cap \TT}$ and $\Prob{\ol{\RR'}\cap \MM'\cap \TT}$ are all at most $e^{-\OM(\om n / \log\log n)}$.

\subsubsection{Connectivity}

We bound $\Prob{\ol{\CC'}\cap \EE'}$. Suppose $\EE'$ holds. Lemma \ref{lem:expansion} implies that in $G = G_{\t_2}$, any $S$ with $|S|\leq \a n$ has $|N_G(S)| \geq 2|S|$. In particular, any connected component has size at least $3\a n$. Suppose $S$ is a set of size $3\a n \leq s \leq n / 2$. Then $e_\G(S, \ol S) \geq s(\b n - s)\geq \e sn$, and for $p\geq \frac{\log n}{n}$,
\al{
  \Prob{\exists\ 3\a n \leq |S| \leq n / 2 : e_G(S,\ol S) = 0} & \leq \sum_{s = 3\a n}^{n / 2} \binom{n}{s} (1-p)^{\e sn}  \\
  & \leq \sum_{s = 3\a n}^{n / 2}\left(\frac{ne}{s} \left(1 - \frac{\log n}{n}\right)^{\e n}\right)^s \\
  & \leq \sum_{s = 3\a n}^{n / 2} \left(\frac{e}{3\a} n^{-\e}\right)^s \leq e^{-\OM(\om n)}.
}

\subsubsection{The set $\mathrm{MEDIUM}$}\label{sec:medium}

Define $\MM'_t$ as the event that the blue subgraph $\G_t^b$ has $|\mathrm{MEDIUM}| \leq \om n / \log\log n$, noting that $\MM' = \MM'_{\t_2}$. We have $\MM'_s\subseteq \MM'_t$ whenever $s\leq t$, so
$$
\ol{\MM'_{\t_2}}\cap \TT \subseteq \ol{\MM'_T}\cap \TT \subseteq \ol{\MM'_T}.
$$

By an argument similar to the one behind Lemma \ref{lem:ptot}, we can couple $\G_T^b$ to $\G_p$ where $p = (p_0 - \om/n)(1 - q) \geq \frac{\log n}{2n}$, and we move to bounding the probability that $\G_p$ has more than $\om n / \log\log n$ vertices of degree less than $\s\log n$.

Let $S$ be a set of $s = \om n / \log\log n$ vertices. For any $v\in S$, the probability that $v$ has degree less than $\s\log n$ in $\G_p$, is at most the probability that it has less than $\s\log n$ edges to $\ol S$. As $v$ has at least $\d(\G) - s$ potential edges to $\ol S$, we have
\al{
  \Prob{e(v, \ol S) < \s\log n} \leq \sum_{k = 0}^{\s\log n} \binom{d_\G(v)}{k} p^{k} \left(1 - p\right)^{\d(\G)-s-k}.
}
Letting $b_k$ denote the summand, we have
$$
\frac{b_{k + 1}}{b_k} = \frac{p}{1-p} \frac{d_\G(v) - k}{k + 1} \geq (1-o(1))\frac{\log n}{2n} \frac{n / 2}{\s\log n + 1} > 2,
$$
as $\s = 1/100$. If follows that $\sum b_k \leq 2b_{\s\log n}$, and as $\frac{\log n}{2n} \leq p \leq 2\frac{\log n}{n}$ and $\d(\G) > (1 + \e)n/2 + 2\s\log n$,
\al{
  \Prob{e(v, \ol S) < \s\log n} & \leq 2\left(\frac{ne}{2\s\log n}\frac{2\log n}{n}\right)^{\s\log n} \left(1 - \frac{\log n}{2n}\right)^{(1+\e)n/2} \\
  & \leq 2\bfrac{e}{\s}^{\s\log n}n^{-(1 + \e)/4}.
}
With $\s = 1/100$, this is at most $n^{-1/8}$. As the $e(v, \ol S)$ are independent random variables for all $v\in S$, the probability that $e(v, \ol S) < \s\log n$ for all $v\in S$ is at most $n^{-s/8}$. It follows that
\al{
  \Prob{|\mathrm{MEDIUM}| > s} & \leq \binom{n}{s}n^{-s/8} \\
  & \leq \left(\frac{ne}{s} \frac{1}{n^{1/8}}\right)^s \\
  & \leq \exp\left\{\frac{7s}{8}\log n - s\log s\right\} \\
  & \leq \exp\left\{- \frac{\om n}{16\log \log n}\right\},
}
as $s = \om n / \log \log n$.

\subsubsection{There are enough red edges outside $\mathrm{MEDIUM}$}

Let $\DD_t$ be the event that at least $\om n$ edges are coloured red in $\G_t$. This is increasing in $t$, and $T = \OM(n\log n)$, so
\begin{multline}
\Prob{\ol{\DD_{\t_2}} \cap \TT} \leq \Prob{\ol{\DD_T} \cap \TT} \leq \Prob{\ol{\DD_T}} \\
= \Prob{\Bin{T}{\frac{\om}{\log n}} < \om n}  \leq e^{-\OM(\om n)}.
\end{multline}
Conditioning on $\DD_{\t_2}$, the probability that there exists some set $M$ of $s = o(n)$ vertices such that more than $\om n / 2$ red edges have at least one endpoint incident to $M$ is at most $e^{-\OM(\om n)}$. As $\MM'$ implies the existence of such a set, we have
$$
\Prob{\ol{\RR'}\cap \MM'\cap \TT} \leq \Prob{\ol{\DD_{\t_2}} \cap \TT} + \Prob{\ol{\RR'}\cap \MM' \mid \DD_{\t_2}} \leq e^{-\OM(\om n)}.
$$

\subsection{Properties of $\G_{\t_2}$}\label{sec:props}

In \eqref{eq:N} we claim that $\Prob{\NN} = \Prob{\SS\cap\EE_1\cap\EE_2} = 1 - o(1)$, which we now prove. 

\subsubsection{$\SS$ -- $\mathrm{SMALL}$ is small}\label{sec:smallsmall}

The set $\mathrm{SMALL}$ is defined as the set of vertices of degree at most $\s\log n$ in $\G_{\t_2}$. We will show that $\mathrm{SMALL}$ is small in $\G_T$, which is enough as $\SS$ is an increasing property. With $p = p_0 - \om/n$, repeating the calculations in Section \ref{sec:medium},
\al{
  \Prob{d_T(v) < \s\log n} \leq 2 \bfrac{2e}{\s}^{\s\log  n} (1-p)^{d_\G(v)}.
}
With $\s = 1/100$ we have $\s\log(2e/\s) < 0.08$, and \eqref{eq:ombelow} shows that $\sum_v (1-p)^{d_\G(v)} = o(n^{0.01})$, so
$$
\mathbb{E}|\mathrm{SMALL}| \leq 2n^{0.08}\sum_v (1-p)^{d_\G(v)} \leq n^{ 0.09}. 
$$
Markov's inequality implies that $|\mathrm{SMALL}| \leq n^{0.1}$ whp.

\subsubsection{$\EE_1$ -- Small sets are sparse}

Let $\G_n$ have $m$ edges and minimum degree $\b n$. Recall, with the threshold $p_0$ as defined in Theorem \ref{thm:2threshold}, that we define for some $\om$,
$$
T = \left(p_0 - \frac{\om}{n}\right)m, \quad T' = \left(p_0 + \frac{\om}{n}\right)m.
$$
By Lemma \ref{lem:threshold}, the hitting time $\t_2$ for having minimum degree $2$ satisfies $T\leq \t_2\leq T'$. In this section we show that $|N_G(S)| \geq 2|S|$ for all $|S| \leq \e n$ whp in $G = \G_{\t_2}$.

\begin{lemma}\label{lem:sparsity}
Suppose $|p - p_0| \leq \om / n$ and let $G = \G_p$. Whp, no $S\subseteq V$ with $|S| \leq 6\a n$ contains more than $\frac{\s\log n}{K}|S|$ edges.
\end{lemma}

\begin{proof}
Recall from Section \ref{sec:threshold} that as $\d(\G_n) \geq n/2$, we have $p_0 \leq 2\frac{\log n}{n}$. The lemma follows from the first moment method: a set $S$ of size $s$ contains at most $\binom{s}{2}$ edges of $\G$, so
\al{
  & \Prob{\exists |S| \leq 6\a n : e_G(S) > \frac{\s\log n}{K}|S|}\\
  \leq\ & \sum_{s = \frac{\s\log n}{2K}}^{\a n}\binom{n}{s}\binom{\binom{s}{2}}{\frac{\s\log n}{K}s}p^{\frac{\s\log n}{K}s} \\
  \leq\ & \sum_{s = \frac{\s\log n}{2K}}^{\a n} \left(\frac{ne}{s}\bfrac{Kes}{2n\s}^{\frac{\s\log n}{K}}\right)^s \\
  \leq\ & \sum_{s = \frac{\s\log n}{2K}}^{\a n} \left(n \bfrac{Ke\a}{2\s}^{\frac{\s\log n}{K}}\right)^s.
}
The constants $K = 10, \s = 1/100, \a = e^{-2000}$ were chosen so that $\frac{\s}{K}\ln(Ke\a/2\s) < -1$, so the summand is $o(1)^s$, and the sum tends to zero.

\end{proof}

\subsubsection{$\EE_2$ -- There are no small structures}

Recall from Section \ref{sec:threshold} that we define $T = (p_0 - \om/n)m$ and $T' = (p_0 + \om / n)m$, with $\om$ as chosen in Section \ref{sec:hit}. The following lemma implies that $\Prob{\ol{\EE_2}} = o(1)$, as $T\leq t\leq T'$ whp by Lemma \ref{lem:threshold}.

\begin{lemma}\label{lem:smallstruct}
Whp, the following holds in $\G_{\t_2}$. No two vertices $u,v\in \mathrm{SMALL}$ are connected by a path of length at most $4$, and no vertex in $\mathrm{SMALL}$ is on a cycle of length at most $4$.
\end{lemma}

\begin{proof}
We let $S$ be the set of vertices of degree less than $\s\log n$ in $\G_T$, noting that $\mathrm{SMALL}\subseteq S$, and bound the probability that $\G_{T'}$ contains a short path or cycle involving $S$ as described. As $T\leq\t_2\leq T'$ whp, the lemma will follow.

We show the proof for the path $P_4$ on three edges, and later explain how the other will follow. Write $\G_t \leftarrow P_4$ for the event that $P_4$ is in $\G_t$ with its two endpoints in $S$. We consider a path $P$ on vertex set $U = \{u_1,u_2,u_3,u_4\}$, where $u_1,u_4$ are the endpoints. Summing over injective maps $\f : U\to V$, writing $v_i = \f(u_i)$, we bound the probability that $v_1,v_4\in S$ and that $G[\f(U)]$ contains three edges as follows. 

Let $p = (p_0 - \om/n)m$ and $p_1 = (p_0 + \om/n)m$. We pick four vertices $v_1,v_2,v_3,v_4$, and three edges forming a path on these vertices. These edges are included in $\G_{p_1}$ with probability $p_1^3$. This  does not significantly change the probability that $v_1$ and $v_4$ are small, and repeating the calculations in Section \ref{sec:smallsmall} we can bound
\al{
  \Prob{\G_{p_1} \leftarrow P_4} & \leq \sum_{v_1,v_2,v_3,v_4} p_1^3\Prob{v_1,v_4\in \mathrm{SMALL}} \\
  & \leq n^2\bfrac{2\log n}{n}^3 \left(\sum_v \Prob{d_T(v) \leq \s\log n}\right)^2 \\
  & \leq 8n^{0.2 - 1}
}
We apply Lemma \ref{lem:ptot}, noting that $(T')^{1/2} = O(n^{1/2}\log n)$, and conclude that
$$
\Prob{\G_{T'} \leftarrow P_4} = O((T')^{1/2} n^{-0.95}) = o(n^{-1/4}),
$$
and $\G_{T'} \nleftarrow P_4$ whp implies that $\G_t\nleftarrow P_4$ for all $T\leq t\leq T'$. In general, suppose $H$ is a small graph on $u$ vertices with $f$ edges, and $s$ vertices required to be in $S$, such that $f + s - u \geq 1$ and $s\leq 2$. This is the case for all the graphs considered, and repeating the above calculations gives
$$
\Prob{\exists T\leq t \leq T':\G_{t} \leftarrow H} = O(T^{1/2}n^{0.05 - f - s + u}) = o(n^{-1/4}).
$$
\end{proof}

\subsection{Expansion: Proof of Lemma \ref{lem:expansion}}\label{sec:expansion}

What now remains is to prove Lemma \ref{lem:expansion}, which states that if $\EE$ then $|N(S)| \geq 2|S|$ whenever $|S| \leq \a n$ for $G_{\t_2}$ and $\G_{\t_2}$. As expansion is an increasing property and $G_{\t_2}\subseteq \G_{\t_2}$, we only need to show that this holds for $G = G_{\t_2}$. We begin with a lemma.

\begin{lemma}\label{lem:largesets}
Suppose $\EE$ occurs and $S\subseteq \mathrm{LARGE}$ has $|S| \leq \a n$. Then $|N_G(S)| \geq 5|S|$ in $G = G_{\t_2}$.
\end{lemma}
\begin{proof}
If $|N_G(S)| < 5|S|$, then $|S\cup N_G(S)| \leq 6\a n$. As $\EE_1\subseteq \EE_1'$, this implies that $e_G(S\cup N_G(S)) \leq \frac{\s\log n}{K}|S|$. We then have
\al{
  |S|\s\log n & \leq \sum_{v\in S} d(v) = 2e_G(S) + e_G(S,N_G(S)) \\
  & \leq e_G(S) + e_G(S\cup N_G(S)) \\
  & \leq \frac{\s\log n}{K}\left(|S| + |S\cup N_G(S)|\right) \\
  & \leq |S|\frac{2\s}{K}\log n + |N_G(S)|\frac{\s\log n}{K},
}
and we conclude that
$$
\frac{|N_G(S)|}{|S|} \geq \frac{K}{\s\log n} \left(\s - \frac{2\s}{K}\right)\log n \geq K - 2.
$$
As $K = 10$, this proves the lemma.
\end{proof}

Now let $S$ be any set of at most $\a n$ vertices, and let $S_1 = S \cap \mathrm{SMALL}$ and $S_2 = S\cap \mathrm{LARGE}$. Then
\begin{multline}
|N_G(S)| = |N_G(S_1)| + |N_G(S_2)|  \\
 - |N_G(S_1)\cap S_2| - |N_G(S_2)\cap S_1| - |N_G(S_1)\cap N_G(S_2)|  \\
\geq |N_G(S_1)| + |N_G(S_2)| - |S_2| - |N_G(S_2)\cap S_1| - |N_G(S_1) \cap N_G(S_2)|.
\end{multline}
We have $|N_G(S_2)| \geq 5|S_2|$ by Lemma \ref{lem:largesets}. As $\EE_2\subseteq \EE_2'$, there are no paths of length $2$ between vertices of $S_1$, which implies $|N_G(S_1)| \geq \sum_{v\in S_1} d(v)\geq 2|S_1|$ (as $\d(G_{\t_2}) \geq 2$), and $|N_G(S_2)\cap S_1| \leq |S_2|$. Finally, as there are no short cycles intersecting $S_1$ and no paths of length $\leq 4$ between vertices of $S_1$, again by $\EE_2$, we have $|N_G(S_1)\cap N_G(S_2)| \leq |S_2|$. This implies that
$$
|N_G(S)| \geq 2|S_1| + 5|S_2| - 3|S_2| \geq 2|S|.
$$

\bibliographystyle{plain}
\bibliography{hamdirac}

\end{document}